\pdfoutput=1
\RequirePackage{ifpdf}
\ifpdf 
\documentclass[pdftex]{sigma}
\else
\documentclass{sigma}
\fi

\numberwithin{equation}{section}

\newtheorem{Theorem}{Theorem}[section]
\newtheorem*{Theorem*}{Theorem}

\newtheorem{Lemma}[Theorem]{Lemma}
\newtheorem{Proposition}[Theorem]{Proposition}

\theoremstyle{definition}
\newtheorem{Definition}[Theorem]{Definition}

\newtheorem{Example}[Theorem]{Example}
\newtheorem{Remark}[Theorem]{Remark}

\begin{document}

\allowdisplaybreaks

\newcommand{\arXivNumber}{2604.17043}

\renewcommand{\PaperNumber}{096}

\FirstPageHeading

\ShortArticleName{Lie Quandles, Leibniz Racks and Noether's First Theorem}

\ArticleName{Lie Quandles, Leibniz Racks\\ and Noether's First Theorem}

\Author{Mohamed ELHAMDADI and Bryce VIRGIN}

\AuthorNameForHeading{M.~Elhamdadi and B.~Virgin}

\Address{Department of Mathematics and Statistics, University of South Florida, Tampa, Florida, USA}
\Email{\mail{emohamed@usf.edu}, \mail{bgvirgin@usf.edu}}
\URLaddress{\url{https://mohamed-elhamdadi.github.io/}}

\ArticleDates{Received May 05, 2026, in final form September 21, 2026; Published online October 08, 2026}

\Abstract{In [\textit{Internat. J. Theoret. Phys.} \textbf{64} (2025), 73, 10~pages], Fritz was motivated by the structure of Hamiltonian/Heisenberg mechanics to define the notion of ``Lie quandle'', which he argued is a~nonlinear generalization of finite-dimensional real Lie algebras. In this article, we will investigate a~linear/nonlinear correspondence to which Fritz's is a~special case, classify a~family of generalizations of these objects, as well as describe some results in the direction of a~nonlinear analogue of Noether's first theorem first described by Fritz.}


\Classification{57K12; 17A32; 22E60; 20N02}

\section{Introduction}\label{intro}
Quandles and racks are algebraic constructs that have arisen from both the study of knots~\cite{EN,Joyce, Matveev} and the study of the conjugation structure of groups~\cite{BES}. Racks are a~concrete realization of the desire to form an algebraic object whose elements can be assigned to automorphisms of itself~\cite{Brieskorn}. Quandles appeared in the literature with many different names. In 1942, Mituhisa Takasaki~\cite{Takasaki} introduced the notion of kei as an abstraction of the notion of symmetric transformation.
The earliest known work on racks is contained in the 1959 correspondence (unpublished) between John Horton Conway and Gavin Wraith who studied
racks in the context of the conjugation operation in a~group. Around 1980, Joyce~\cite{Joyce} and Matveev~\cite{Matveev} independently introduced the notion of a~quandle. To each oriented knot, they associated a~quandle in such a~way that the classical problem of determining knot equivalence is translated into the problem of quandle isomorphism.

Since their introduction, quandles have been the subject of extensive study in both topology and algebra. From an algebraic perspective, they are naturally related to a~range of algebraic structures, including Lie algebras~\cite{CCES1, CCES2}, Frobenius algebras~\cite{CCEKS} and the Yang--Baxter equation~\cite{BES}, Hopf algebras~\cite{And-Grana, CCES2}, quasigroups and Moufang loops~\cite{Elhamdadi}, and ring theory~\cite{EFT}.
Quandles have also been used in the construction of invariants for handlebody knots by enriching their structure with a~group-indexed family of operations, leading to the notion of a~$G$-family of quandles introduced in~\cite{IIJO}. This enhanced framework allows for the development of stronger invariants for knots, links, and related objects such as spatial graphs.

Topological quandles were investigated in~\cite{Rubinsztein} and their cohomology was studied in~\cite{EM, ESZ}. Quandles and Lie racks, objects partially solving the coquecigrue problem, were investigated in~\cite{Kinyon}. In~\cite{Covez}, the author showed that
Leibniz algebras integrate (locally) to local Lie racks, providing a~precise solution to the coquecigrue problem~\cite{Loday}. In~\cite{LucTa}, the author showed that affine algebraic racks provide a~solution to a~version of the well-known coquecigrue problem for left and right Leibniz algebras over arbitrary fields.

Tobias Fritz introduced the notion of Lie quandle in~\cite{Fritz} as a~nonlinear generalization of finite-dimensional real Lie algebras. Fritz was motivated to consider this generalization by the way observables in both the Hamiltonian formulation of classical mechanics and the Heisenberg formulation of quantum mechanics generate 1-parameter subgroups of automorphisms of the entire observable algebra. This feature of these frameworks led to a~particularly simple way of stating Noether's first theorem, which necessarily leads to an equivalent statement of Noether's first theorem for Lie quandles. The Lie quandles section in~\cite{Fritz} concludes with some examples of Lie quandles both satisfying Noether's first theorem and not, and a~suggestion that connectedness may be a~sufficient condition to imply Noether's first theorem on a~Lie quandle. In~\cite{Hayami}, Hayami provides a~family of examples of Lie quandles which arise from the gauge symmetries of principal bundles, and shows that Noether's first theorem holds true for all of these examples.

In this article, we describe some of the relationships between some categories related to Lie quandles and Lie algebras, discuss a~special case classification result of one of these categories, and investigate the necessity or sufficiency of possible hypotheses for
further generalizations of Noether's first theorem in addition to Fritz's suggestion of connectedness.

The paper is organized as follows: In Section~\ref{Sec2}, we review basic definitions, including those of Lie quandles and
$G$-families of quandles, and introduce new terminology, such as the notion of a~smooth $G$-family of smooth racks. In Section~\ref{Sec3}, we show that there is a~categorical embedding of the category of finite-dimensional real Leibniz algebras into the category of smooth $\mathbb{R}$-families of smooth racks which is analogous to the embedding of the category of finite-dimensional real vector spaces into the category of smooth manifolds. This correspondence can be specialized to the particular case of finite-dimensional real Lie algebras and Lie quandles, as suggested by Fritz. We also demonstrate that the category of smooth simply connected Lie group families of smooth racks can be embedded into the category of smooth Lie algebra families of smooth racks. Section~\ref{Sec4} provides a~specialized classification of a~simple class of smooth $G$-families of smooth racks. We conclude with Section~\ref{Sec5}, in which we discuss Noether's first theorem in the context of smooth $G$-families of smooth racks, and among these results we see that connectedness is not a~necessary condition.

\section{Review of racks and Leibniz algebras}\label{Sec2}
We start by briefly recalling some definitions and examples.

\begin{Definition}
 A~rack is a~pair $(Q,\triangleleft)$, where $Q$ is a~set and $\triangleleft\colon Q^2 \to Q$ is a~map satisfying the following two axioms:
\begin{enumerate}\itemsep=0pt
 \item[(1)]
 The right multiplication $R_x\colon Q \to Q, y \mapsto y \triangleleft x$, is a~bijection,

\item[(2)]
For all $p,q,r \in Q$, $ (p \triangleleft q )\triangleleft r = (p \triangleleft r )\triangleleft (q \triangleleft r)$.
\end{enumerate}
If $p \triangleleft p=p$ holds for all $p \in Q$, then $(Q,\triangleleft)$ is called a~quandle.
\end{Definition}

\begin{Example}
For any set $Q$, the function $\triangleleft\colon Q^2 \to Q$ defined by $p \triangleleft q = p$ for all $p,q \in Q$ makes the pair $(Q,\triangleleft)$ into a~quandle (known as a~trivial quandle).
\end{Example}

\begin{Example}
 Given a~set $X$, a~group $G$ acting on $X$, and a~map $\partial \colon X \to G$ satisfying~${
 \partial (gx) = g\partial(x) g^{-1}}$ for all $ x \in X $ and $g \in G$,
 we can define the operation $\triangleleft\colon X^2\to X $ via $x\triangleleft y = \partial(y) x$, which makes $(X,\triangleleft)$ into a~rack. Racks belonging to this family are known as augmented racks, and quandles belonging to this family will be called augmented quandles.
\end{Example}

\begin{Remark}
 As we will see in Section~\ref{Sec3}, a~natural augmented quandle structure exists on any Lie algebra, with the group action given by the adjoint representation of its Lie group, and $\partial$ given by the associated exponential map. Examples of these objects arising from gauge symmetries in physics, accordingly called gauge quandles, can be found in Hayami's recent article~\cite{Hayami}.\looseness=-1
\end{Remark}

 Typical examples of quandles include the following:
 \begin{itemize}\itemsep=0pt
\item A~conjugacy class $Q$ of a~group $G$ with the quandle operation $p \triangleleft q=q^{-1}pq $ is called a~{\it conjugation quandle}.
\item A~group $G$ endowed with the binary operation $\triangleleft$ defined by $p \triangleleft q=q p^{-1}q $ is called a~{\it core quandle}.
\item
 A~{\it generalized Alexander quandle} is defined by a~pair $(G, f)$ where
$G$ is a~group and $f \in {\rm Aut}(G)$,
and the rack operation is defined by
$p\triangleleft q  = qf\bigl(q^{-1}p\bigr)$.
If $G$ is abelian, this is called an {\it Alexander quandle}.
\end{itemize}

\begin{Remark} We note that if $(Q,\triangleleft)$ is a~core quandle or a~generalized Alexander quandle, then for all $p,q \in Q$ we have that $p \triangleleft q = p$ if and only if $q \triangleleft p = q$.
\end{Remark}

\begin{Definition}[{\cite{Fritz}}] \label{LieQuandle}
 A~Lie quandle is a~pair $(Q,\triangleleft)$ where $Q$ is a~smooth manifold and $\triangleleft\colon Q^2\times \mathbb{R} \to Q$ is a~smooth map having the following three properties:
\begin{enumerate}\itemsep=0pt
 \item[(1)]
 For all $q \in Q$ and $ t \in \mathbb{R}$, $q \triangleleft_t q =q$,

\item[(2)]
For all $q,p \in Q$ and $t,s \in \mathbb{R}$, $(q \triangleleft_t p )\triangleleft_s p = q \triangleleft_{t+s} p $ and $q\triangleleft_0 p = q$,

\item[(3)]
For all $ t,s \in \mathbb{R}$
and $q,p,r \in Q$, $(q\triangleleft_t p) \triangleleft_s r = (q \triangleleft_s r)\triangleleft_{t} ( p\triangleleft_s r) $.
\end{enumerate}
\end{Definition}

\begin{Definition}
 A~morphism from a~Lie quandle $(Q,\triangleleft)$ to another $(Q',\triangleleft')$ is a~smooth map $f\colon Q \to Q'$ and a~real number $a \in \mathbb{R}$ such that $
 f(p\triangleleft_t q) = f(p)\triangleleft'_{at} f(q)$ for all $p,q \in Q $ and all~${t \in \mathbb{R}}$.
 The category of Lie quandles and Lie quandles morphisms will be denoted $\textbf{LiQd}$.
\end{Definition}

Definition~\ref{LieQuandle} turns out to be a~special case of what is known as $G$-family of quandles~\cite{IIJO} in the quandle literature.

\begin{Definition}[{\cite{IIJO}}] A~$G$-family of quandles for a~given group $G$ is a~pair $(Q,\triangleleft)$, where $Q$ is a~set and $\triangleleft\colon Q^2\times G \to Q$ is a~function satisfying:
\begin{enumerate}\itemsep=0pt
 \item[(1)]
 $q \triangleleft_g q =q$ for all $q \in Q$ and $g \in G$.

\item[(2)]
$ (q \triangleleft_g p )\triangleleft_h p = q \triangleleft_{gh} p $ and $q\triangleleft_e p = q$ for all $q,p \in Q$.

\item[(3)]
$ (q\triangleleft_g p) \triangleleft_h r = (q \triangleleft_h r)\triangleleft_{h^{-1}gh} (p\triangleleft_h r) $ for all $g,h \in G$ and $q,p,r \in Q$.
\end{enumerate}
\end{Definition}

We will denote the $G$-family of quandles by $(Q, \{\triangleleft_g\}_{g \in G})$. To make the notation shorter, we will use $(Q,G)$ instead.

\begin{Remark}
 Notice that if we set $p=r$ in Axiom (3), we obtain $ {(q\triangleleft_g p) \triangleleft_h p = (q \triangleleft_h p)\triangleleft_{h^{-1}gh} p}$. Now if we denote $R_p^g(q):=q\triangleleft_g p$ then Axiom~(2) states that for each fixed $p$, the family~$\{R_p^g\}_{g \in G}$ is a~representation of the group $G$. While Axiom~(3) says that these representations are equivariantly conjugate under the translation $R_r^h$. In other words, moving the base point from~$p$ to~$R_r^h(p)$ forces the group parameter to be conjugated from $g$ to $h^{-1}gh$.
\end{Remark}

The following are a~few examples of $G$-families of quandles.
 \begin{itemize}\itemsep=0pt
 \item
For any group $G$ and any set $Q$, defining $p \triangleleft_g q=p$ for all $p,q \in Q$ and all $g \in G$. This gives a~$G$-family of quandles called the \emph{trivial} $G$-family of quandles.

 \item
Let $(Q, \triangleleft )$ be a~quandle, and define $p \triangleleft_iq={R_q}^i(p)$ then $(Q,\mathbb{Z})$ is a~$\mathbb{Z}$-family of quandles. If~$(Q,\triangleleft)$ also has the property that ${R_q}^{n}(p)=p$ for all $p,q \in Q$, then this $\mathbb{Z}$-family of quandles induces a~$\mathbb{Z}_{n}$-family of quandles.

 \item
For a~ring $R$ and a~group $G$, let $R[G]$ be the group ring. Then any (right) $R[G]$-module $Q$ becomes a~$G$-family of quandles by defining $p \triangleleft_g q:=pg+q(1-g)$.
 \end{itemize}

 Notice that a~$G$-family of quandles $(Q,G)$ induces a~quandle operation on $Q \times G$ given by
 \[
 (p,g)*(q,h):= \bigl(p \triangleleft_g q, h^{-1}gh\bigr).
 \]

 We are grateful to Masahico Saito for observing that Lie quandles can be viewed as
instances of $\mathbb{R}$-families of quandles internal to the category of smooth manifolds (henceforth denoted by~$\textbf{Mfd}$). Given this observation, it is natural to internalize the general notion of $G$-families of quandles to~$\textbf{Mfd}$, and to subject these new objects to Fritz' line of questioning.

\begin{Definition}
 We define a~\textit{smooth} $G$-\textit{family of smooth racks} for a~given (real) Lie group $G$ to be a~pair $(R,\triangleleft)$, where $R$ is a~smooth manifold and $\triangleleft\colon R^2\times G \to R$ is a~smooth map such that for all $p,q,r \in R$ and all $g,h \in G$, we have that
\begin{enumerate}\itemsep=0pt
\item[(1)]
$q\triangleleft_e p = q$.
 \item[(2)]
 $ (q \triangleleft_g p )\triangleleft_h p = q \triangleleft_{gh} p$.

\item[(3)]
$ (q\triangleleft_g p) \triangleleft_h r = (q \triangleleft_h r)\triangleleft_{h^{-1}gh} (p\triangleleft_h r) $.
\end{enumerate}
\end{Definition} A~smooth $G$-family of smooth racks $(R,\triangleleft)$ is called a~smooth $G$-family of smooth quandles if $q\triangleleft_g q =q$ for all $q \in R$ and $g \in G$. A~smooth $\mathbb{R}$-family of smooth racks is called a~\textit{Leibniz rack}.

\begin{Remark}
 The name Leibniz rack was chosen to indicate that Leibniz racks have the same relationship to Leibniz algebras as Lie quandles have to Lie algebras, a~fact which will be explored in Section~\ref{Sec3}. One should be mindful not to confuse Lie quandles or Leibniz racks with Lie racks, objects defined by Kinyon~\cite{Kinyon} as a~partial solution to the coquecigrue problem.
\end{Remark}

\begin{Example}
 For any smooth manifold $R$ and any Lie group $G$, we can define the trivial smooth $G$-family of smooth rack operations via $p \triangleleft_g q=p$ for all $p,q \in R$ and all $g \in G$. If~${G =\mathbb{R}}$, we call this family of operations the trivial Leibniz rack operation for $R$.
\end{Example}

\begin{Example}
 All Lie quandles are Leibniz racks.
\end{Example}

\begin{Example}
 Given Lie groups $F$, $G$, a~closed Lie subgroup $H \subset F$, and any Lie group action~${g \mapsto \rho_g \in \text{Aut}(F)}$ such that $\rho_g $ fixes $H$ pointwise for each $g \in G$, the operation $\triangleleft\colon (F/H)^2\times G \to F/H$ defined by $f'H \triangleleft_g fH = f\rho_{g-1}\bigl(f^{-1}f'\bigr)H$ for all $fH,f'H \in F/H$ is a~smooth $G$-family of smooth quandle operations on the coset manifold $F/H$. When $H = \{e\}$, we say that the smooth $G$-family of smooth quandles $(F,\triangleleft)$ is a~smooth $G$-family of generalized Alexander quandles on $F$. 
 In Section~\ref{Sec4}, we classify the subcollection defined by further requiring that $F$ and $G$ are Euclidean spaces.
\end{Example}

\begin{Remark}
In Section~\ref{Sec5}, we will see that a~special subset of the family of smooth $G$-families of smooth racks in the preceding definition obeys Noether's first theorem. It is worth noting that for all $g \in G$, the pairs $(F/H,\triangleleft_g)$ are homogeneous quandles (see~\cite[Theorem 7.1]{Joyce} for justification that homogeneous quandles have this form) internal to the category of smooth manifolds.
\end{Remark}

\begin{Example}
The parameterized gauge quandles $\bigl(P,\triangleleft^X\bigr)$ of~\cite{Hayami} are a~family of Lie quandles which obey Noether's first theorem.
\end{Example}

\begin{Definition}
 A~morphism between two smooth Lie group families of smooth racks $(R,\triangleleft)$ and $(R',\triangleleft')$ is a~smooth map $f\colon R\to R'$ and a~Lie group homomorphism $\varphi\colon G \to G'$ such that for all $p,q \in R$ and $g \in G$, we have
 \[ f(p \triangleleft_g q) = f(p) \triangleleft_{\varphi(g)}'f(q).\]
To indicate the relationship between $f$ and $\varphi$, we say that $f$ is a~morphism over $\varphi$. The category of smooth Lie group families of smooth racks and morphisms as defined above will be denoted~$\textbf{GpRk}$, the category of smooth Lie group families of smooth quandles is denoted $\textbf{GpQd}$, and the category of Leibniz racks and morphisms will be denoted $\textbf{LeRk}$.
\end{Definition}

Due to the correspondence between simply connected Lie groups and finite-dimensional real Lie algebras, the subcategory of $\textbf{GpRk}$ of objects whose operations are parameterized by simply connected Lie groups can be embedded into the category whose objects are families of smooth racks parameterized by the elements of a~Lie algebra, as we will see in Section~\ref{Sec3}.

 Leibniz algebras were introduced by Loday~\cite{Loday} as a~non-commutative version of Lie algebras, where the bracket is not necessarily antisymmetric. We recall the definition.

\begin{Definition} A~(right) Leibniz algebra is a~vector space $A$ together with a~bilinear brack\-et~$[\cdot,\cdot]$ which has the (right) Leibniz property
\[ [[a,b],c] = [[a,c],b]+[a,[b,c]]. \]
Every Lie algebra is a~Leibniz algebra, and for both Lie and Leibniz algebras we adopt the convention that for each $a \in A$ the map ${\rm ad}_a\colon A \to A$ is the inner derivation defined by ${\rm ad}_a(b) = [b,a]$ for each $b \in A$. We note that this choice for the adjoint comes from evaluation of the right argument, as opposed to the left, which is the standard choice for Lie algebras. This convention is more convenient for us as a~byproduct of having chosen the right-acting convention for our definitions of rack and quandle.
\end{Definition}

\begin{Example}
 It is clear that a~Lie algebra is an example of Leibniz algebra.
\end{Example}

\begin{Example}[a non-Lie Leibniz algebra]\label{nonLie}
Let $L = \mathrm{span}\{e_1,e_2\}$ with
$[e_1,e_1] = e_2$, and with all other brackets being equal to zero.
Then $L$ is a~Leibniz algebra but not a~Lie algebra.
\end{Example}

\begin{Example}
 Given a~Lie algebra $(\mathfrak{g},[\cdot,\cdot])$ and a~derivation $D\colon \mathfrak{g}\to \mathfrak{g}$ with $D^2=0$. Then~${[\cdot,\cdot]_D}$ defined by $[u,v]_D = [u,Dv]$ for all $u,v \in \mathfrak{g}$ makes the pair $(\mathfrak{g},[\cdot,\cdot]_D)$ a~Leibniz algebra.
\end{Example}

\begin{Example}[hemisemidirect product]
Let $\mathfrak{g}$ be a~Leibniz algebra acting on a~vector space~$V$. Define
\[
[(x,v),(y,w)] = ([x,y], x \cdot w - y \cdot v).
\]
This gives a~Leibniz algebra structure on $\mathfrak{g} \oplus V$.
\end{Example}

\begin{Definition}
 Let $A$ and $B$ be two Leibniz algebras. A~Leibniz algebra homomorphism is a~linear map $f\colon A \rightarrow B$ that satisfies
$f([x,y])=[f(x),f(y)] $ for all $x,y \in A$.
 We denote the category of finite-dimensional real Leibniz algebras and morphisms by $\textbf{LeAg}$. The full subcategory whose objects consist of all Lie algebras in $\textbf{LeAg}$ will be denoted $\textbf{LiAg}$, and the category of finite-dimensional real vector spaces and linear maps will be denoted $\textbf{Vec}_{\mathbb{R}}$.
\end{Definition}

\begin{Definition}
 We define a~smooth $A$-family of smooth racks for a~given finite-dimensional real Leibniz algebra $A$ to be a~pair $(R,\triangleleft )$, where $R$ is a~smooth manifold and $\triangleleft\colon R^2\times A \to R$ is a~smooth map (with evaluation denoted $(q,p,a) \mapsto q\triangleleft_{a}p$) such that for all $p,q,r \in R$, all~${t,s \in \mathbb{R}}$, and all $a,b \in A$, we have that
\begin{enumerate}\itemsep=0pt
 \item[(1)]
 $p \triangleleft_0 q=p$.

 \item[(2)] $(p\triangleleft_{ta}q)\triangleleft_{sa}q = p \triangleleft_{(t+s)a}q $.

\item[(3)] $ (p \triangleleft_a q) \triangleleft_b r = (p\triangleleft_b r) \triangleleft_{e^{{\rm ad}_b}(a)} (q \triangleleft_b r). $
\end{enumerate}
We say that a~smooth $A$-family of smooth racks is a~smooth Lie algebra family of smooth racks if the Leibniz algebra $A$ is a~Lie algebra.
\end{Definition}

\begin{Example}
 As we will see in Section~\ref{Sec3}, any smooth $G$-family of smooth racks $(R,\triangleleft)$, gives rise to the pair $(R,\blacktriangleleft)$, where $\blacktriangleleft\colon R^2\times \mathfrak{g} \to R$ is defined by $ p\blacktriangleleft_vq = p \triangleleft_{\operatorname{exp}(v)} q $ for all $p,q \in R$ and all $v \in \mathfrak{g}$.
\end{Example}

\begin{Remark}
 For any smooth $A$-family of smooth racks $(R,\triangleleft)$, and any $a \in A$ the operation~${\triangleleft^a\colon R^2 \times \mathbb{R}\to R}$ defined by $p\triangleleft^a_tq = p\triangleleft_{ta}q $ for all $p,q \in R$ and $t \in \mathbb{R}$ makes the pair $\bigl(R,\triangleleft^a\bigr)$ a~Leibniz rack. If $(R,\triangleleft)$ is an $A$-family of smooth quandles, then $\bigl(R,\triangleleft^a\bigr)$ is a~Lie quandle for each $a \in A$.
\end{Remark}

\begin{Definition} A~morphism between two smooth Leibniz algebra families of smooth racks $(R, \triangleleft )$ and $(R',\triangleleft')$ consists of a~smooth map $f\colon R\to R'$ and a~Leibniz algebra morphism $\varphi\colon A \to A'$ such that
$
f( p \triangleleft_{a} q) = f(p) \triangleleft'_{\varphi(a)} f(q)$ for all $q,p \in R$, and all $a\in A$.
 The category of smooth Leibniz algebra families of smooth racks and morphisms will be denoted~$\textbf{LeFRk}$, the category of smooth Lie algebra families of smooth racks (with the same morphisms) will be denoted $\textbf{LiFRk}$, and the category of smooth Lie algebra families of smooth quandles will be denoted $\textbf{LiFQd}$.
\end{Definition}

\section{Leibniz algebras and Leibniz racks} \label{Sec3}

In~\cite{Fritz}, Fritz showed how the Lie bracket of a~Lie algebra may be recovered from the natural Lie quandle operation associated to it. His explanation formed the core of an argument that~$\textbf{LiQd}$ has the same relationship to $\textbf{LiAg}$ that~$\textbf{Mfd}$ has to $\textbf{Vec}_\mathbb{R}$. This turns out to be a~special case of a~more general correspondence between $\textbf{LeAg}$ and $\textbf{LeRk}$. The proof of the following lemma exists elsewhere in the literature (see, for example, \cite[p.~106]{Kinyon}), we simply include it for the benefit of this article's self-containment.

\begin{Lemma}
 For any finite-dimensional real Leibniz algebra $A$, and any $a \in A$, we have that~${e^{{\rm ad}_a} \colon A \to A }$ is an automorphism of Leibniz algebras.
 \end{Lemma}

\begin{proof} Let $A$ be a~finite-dimensional real Leibniz algebra and $a,b,c \in A$.
 We begin by defining the smooth functions $f,g \colon \mathbb{R}\to A$ via
 \[
 f(t) = e^{t {\rm ad}_a}([b,c]), \qquad
 g(t) = \bigl[e^{t {\rm ad}_a}(b),e^{t {\rm ad}_a}(c)\bigr].
 \]

We note that $f(0) = [b,c] = g(0)$, and furthermore that
\begin{gather*}
f'(t) = \bigl[e^{t {\rm ad}_a}([b,c]),a\bigr] = [f(t),a], \\
g'(t) = \bigl[\bigl[e^{t a}(b),a\bigr],e^{t {\rm ad}_a}(c)\bigr]+\bigl[e^{t {\rm ad}_a}(b),\bigl[e^{t {\rm ad}_a}(c),a\bigr]\bigr]
=\bigl[\bigl[e^{t {\rm ad}_a}(b),e^{t {\rm ad}_a}(c)\bigr],a\bigr] = [g(t),a],
\end{gather*}
hence $f$ and $g$ are both solutions to the same first order linear initial value problem (IVP) $u' = [u,a]$ with $u(0) = [b,c]$, hence by the uniqueness of solutions to first order linear IVPs we must have that $f=g$. By setting $t=1$, we learn that $
 e^{ {\rm ad}_a}([b,c]) = \bigl[e^{ {\rm ad}_a}(b),e^{ {\rm ad}_a}(c)\bigr]$,
furthermore, $e^{{\rm ad}_a}$ has $e^{-{\rm ad}_a}$ as an inverse, and therefore $e^{{\rm ad}_a}$ is a~Leibniz algebra automorphism.
\end{proof}

As a~simple corollary to the lemma above, for any Leibniz algebra $A$ and any~${a,b,c \in A}$, we have that
\[
\bigl(e^{{\rm ad}_a} \circ {\rm ad}_b \circ e^{-{\rm ad}_a}\bigr)(c) = e^{{\rm ad}_a}\bigl(\bigl[e^{-{\rm ad}_a}(c),b\bigr]\bigr) = \bigl[c,e^{{\rm ad}_a}(b)\bigr] = {\rm ad}_{e^{{\rm ad}_a}(b)}(c).
\]
This implies that \smash{$e^{{\rm ad}_a} \circ {\rm ad}_b \circ e^{-{\rm ad}_a} = {\rm ad}_{e^{{\rm ad}_a}(b)}$},

\begin{Proposition}
The category $\textbf{LeAg}$ is isomorphic to the subcategory of $\textbf{LeRk}$ consisting of all of its objects $(V,\triangleleft)$ such that $V \in \textbf{Vec}_\mathbb{R}$ and so that the map
$V^2 \to V $ defined by $ (u,v) \mapsto \frac{{\rm d}}{{\rm d} t}|_{t=0} (u \triangleleft_t v)$ for all $u,v \in V$
is $\mathbb{R}$-bilinear, and of all $\textbf{LeRk}$-morphisms which lie over the identity of $\mathbb{R}$ and are also linear maps. The functor constructed to show this can be restricted to Fritz's isomorphism of the category of $\textbf{LiAg}$ to the analogous subcategory of $\textbf{LiQd}$.
\end{Proposition}

\begin{proof}
 For any finite-dimensional real Leibniz algebra $A$, define $F(A)$ to the pair $(A,\triangleleft)$, where ${a\triangleleft_s b = e^{s {\rm ad}_b}(a)}$ for all $a,b \in A$ and $s \in \mathbb{R}$. Since we will need it later, we note that
 \[
 \frac{{\rm d}}{{\rm d} s}\bigg|_{s=0} (a\triangleleft_s b) = \frac{{\rm d}}{{\rm d} s}\bigg|_{s=0} e^{s {\rm ad}_b}(a) = [a,b].
 \]

Now to see that $F(A)$ is a~smooth $\mathbb{R}$-family of smooth racks, we first check that
\[
a\triangleleft_0 b = e^{0 {\rm ad}_b}(a) = a
 \]
before subsequently noting
\[
(a \triangleleft_t b ) \triangleleft_s b = e^{s {\rm ad}_b}e^{t {\rm ad}_b}a = e^{(t+s) {\rm ad}_b}a = a \triangleleft_{t+s}b
\]
hence $F(A)$ satisfies the first property. Next, we see that
\begin{align*}
 (a \triangleleft_t b ) \triangleleft_s c &=e^{s {\rm ad}_c}e^{t {\rm ad}_b}a=e^{s {\rm ad}_c}e^{t {\rm ad}_b}e^{-s {\rm ad}_c}e^{s {\rm ad}_c}a \\
 &=e^{t e^{s {\rm ad}_c}{\rm ad}_b e^{-s {\rm ad}_c}} (a\triangleleft_s c) =e^{t {\rm ad}_{b\triangleleft_s c}} (a\triangleleft_s c) =(a\triangleleft_s c) \triangleleft_t (b\triangleleft_s c).
 \end{align*}

Where we have used the fact that $e^{{\rm ad}_a}{\rm ad}_b e^{-{\rm ad}_a} = {\rm ad}_{e^{{\rm ad}_a}(b)} $ discussed previously. Note that if~$A$ is a~Lie algebra, so that $[a,a]=0$ for all $a \in A$, then we would also find that
\[
 a \triangleleft_s a = e^{s{\rm ad}_a}a = a
 \]
making $F(A)$ a~Lie quandle in this case.

We note that for $F(A)=F(B) =(X,\triangleleft)$, we must have that $A$ and $B$ share the same underlying set of vectors $X$. We also find that for all $x,y \in X$
\[
[x,y]_A = \frac{{\rm d}}{{\rm d} s}\bigg|_{s=0} (x \triangleleft_s y) = [x,y]_B.
\]
Hence $A=B$, and so $A \mapsto F(A)$ is an injective function from the set of finite-dimensional real Leibniz algebras to the set of Leibniz racks.

 For any Leibniz algebra homomorphism $f\colon A \to B$, we can see directly that
\begin{align*}
 f\bigl(a\triangleleft^A_s b\bigr) &= f\big( e^{s {\rm ad}_b^A}(a)\big)= f\left( \sum_{j=0}^{\infty} \frac{s^j}{j!} \bigl({\rm ad}_b^A\bigr)^j(a)\right)= \sum_{j=0}^{\infty} \frac{s^j}{j!} f\bigl( \bigl({\rm ad}_b^A\bigr)^j(a)\bigr)\\
&= \sum_{j=0}^{\infty} \frac{s^j}{j!} \bigl({\rm ad}_{f(b)}^B\bigr)^j(f(a)) = e^{s {\rm ad}_{f(b)}^B}(f(a)) = f(a)\triangleleft^B_sf(b)
.
\end{align*}
Hence $F(f)=f$ (which clearly preserves identities and compositions) leads us to conclude that~${F\colon \textbf{LeAg} \to \textbf{LeRk}}$ is a~functor. Moreover, $F$ is manifestly faithful, and so $F$ is a~functor that is both injective on objects and faithful.

Now suppose $(V,\triangleleft')$ is a~Leibniz rack such that $V \in \textbf{Vec}_\mathbb{R}$ and so that the map $[\cdot,\cdot]\colon V^2 \to V$ defined by
\[
 [u,v] = \frac{{\rm d}}{{\rm d} t}\bigg|_{t=0} (u \triangleleft_t' v)
 \]
is $\mathbb{R}$-bilinear. These assumptions clearly leave us with an $\mathbb{R}$-algebra $(V,[\cdot,\cdot])$, and by the group law and generalized right distributivity axioms for Leibniz racks, we find that
\begin{align*}
 [[u,v,],w] &= \left[ \frac{{\rm d}}{{\rm d} t} (u \triangleleft_t' v),w\right]_{t=0} =\left( \frac{{\rm d}}{{\rm d} t} \left[ (u \triangleleft_t' v),w\right] \right)_{t=0}
 =\left(\frac{{\rm d}}{{\rm d} t} \frac{{\rm d}}{{\rm d} s} ( (u \triangleleft_t' v) \triangleleft_s' w ) \right)_{t,s=0} \\
 &= \left(\frac{{\rm d}}{{\rm d} s}\frac{{\rm d}}{{\rm d} t} ( (u\triangleleft_s' w) \triangleleft_t' (v \triangleleft_s' w)) \right)_{t,s=0}
 = \left(\frac{{\rm d}}{{\rm d} s} [u\triangleleft_s' w, v \triangleleft_s' w] \right)_{s=0}\\
 &= \left[\frac{{\rm d}}{{\rm d} s}( u\triangleleft_s' w), v \triangleleft_s' w\right] _{s=0}+ \left[u\triangleleft_s' w, \frac{{\rm d}}{{\rm d} s}(v \triangleleft_s' w)\right]_{s=0} = [[u,w],v]+ [u,[v,w]].
 \end{align*}
Hence $(V[\cdot,\cdot])$ is an object in $\textbf{LeAg}$. Now consider $F(V,[\cdot,\cdot])=(V,\triangleleft)$, and note that since for each $v \in V$ the maps $t \mapsto \triangleleft_t v, \triangleleft_t' v$ are both one parameter groups of diffeomorphisms of $V$ such that for every $u \in V$
\[
 \frac{{\rm d}}{{\rm d} t} \bigg|_{t=0} (u \triangleleft_tv) =\frac{{\rm d}}{{\rm d} t} \bigg|_{t=0} e^{t {\rm ad}_v}(u) = {\rm ad}_v(u) = [u,v] = \frac{{\rm d}}{{\rm d} t} \bigg|_{t=0} (u \triangleleft_t'v).
 \]
By the uniqueness of one parameter groups associated to a~given vector field, we must have that $\triangleleft' = \triangleleft$, and so $F(V,[\cdot,\cdot]) = (V,\triangleleft')$. Hence, every such object $(V,\triangleleft')$ of $\textbf{LeRk}$ belongs to the image of $F$'s object component. It is manifest from $F$'s construction that its image is also confined to this set of objects, and so all that remains is to ensure that $F$'s image contains the right morphisms.

It is quite clear that $F$'s morphism part can only output morphisms over $\operatorname{Id}\colon\mathbb{R}\to \mathbb{R}$ which are also linear maps. In order to see that all $\textbf{LeRk}$ morphisms over $\mathrm{Id}\colon \mathbb{R}\to \mathbb{R}$ which are also linear maps lie in~$F$'s image, let $\bigl(V,\triangleleft^V\bigr)=F\bigl(V,[\cdot,\cdot]^V\bigr)$ and $\bigl(U,\triangleleft^U\bigr)=F\bigl(U,[\cdot,\cdot]^U\bigr)$, and let $f\colon U \to V$ be a~$\textbf{LeRk}$ morphism over $\operatorname{Id}\colon\mathbb{R}\to \mathbb{R}$ which is also a~linear map. We then calculate
\begin{align*}
f \bigl( [u,u']^U \bigr) &= f \left( \frac{{\rm d}}{{\rm d} t}\bigg|_{t=0} \bigl(u\triangleleft_t^U u'\bigr) \right) = \frac{{\rm d}}{{\rm d} t}\bigg|_{t=0} f\bigl(u \triangleleft_t^U u'\bigr)= \frac{{\rm d}}{{\rm d} t}\bigg|_{t=0} f(u) \triangleleft_t^V f(u') \\
& = [f(u),f(u')]^V.
 \end{align*}
Hence $f$ clearly belongs to the domain of $F$'s morphism components, and consequently it must also belong to its image by the definition of $F$'s morphism component.
\end{proof}

The existence of this categorical embedding indicates that $\textbf{LeRk}$ has the same relationship to $\textbf{LeAg}$ that $\textbf{Mfd}$ has to $ \textbf{Vec}_\mathbb{R}$. The same is true of $\textbf{LiQd}$ with respect to $\textbf{LiAg}$, as Fritz~\cite{Fritz} suggested.

We will conclude this section with a~demonstration of an embedding of the category of smooth simply connected Lie group families of smooth racks into $\textbf{LiFRk}$, following from the Lie--Cartan theorem.

We have the following proposition.

\begin{Proposition}There exists a~functor from the category of smooth $($simply-connected$)$ Lie group families of smooth racks into $\textbf{LiFRk}$ which is fully-faithful, and strictly injective on objects.
\end{Proposition}

\begin{proof}
 For any Lie group $G$ and any smooth $G$-family of smooth racks $(R,\triangleleft)$, we can define~${F(R,\triangleleft) = (R,\blacktriangleleft )}$, where $\blacktriangleleft \colon R^2\times \mathfrak{g} \to R$ is defined via
$
 r {\blacktriangleleft }_v r' =r \triangleleft_{\operatorname{exp}(v)} r' $ for all $r,r' \in R$, and all $v \in \mathfrak{g}$,
 where $\mathfrak{g}$ is $G$'s Lie algebra and $\operatorname{exp}\colon \mathfrak{g} \to G$ is $G$'s exponential map. Now to demonstrate that~$F(R,\triangleleft)$ defined above is a~smooth $\mathfrak{g}$-family of smooth racks, we first note that for all $p,q \in R$ we must have that
 \[ p \blacktriangleleft_0 q = p \triangleleft_{\operatorname{exp}(0)} q = p \triangleleft_e q = p,\]
 which ensures that $F(R,\triangleleft)$ obeys the first requirement of a~smooth $\mathfrak{g}$-family of smooth racks. Next we see that
 \begin{align*}
 (p {\blacktriangleleft}_{tv} q){\blacktriangleleft}_{sv} q &= (p \triangleleft_{\operatorname{exp}(tv)} q)\triangleleft_{\operatorname{exp}(sv)} q
 = p \triangleleft_{\operatorname{exp}(tv)\operatorname{exp}(sv)} q = p \triangleleft_{\operatorname{exp}((t+s)v)} q = p{\blacktriangleleft}_{(t+s)v} q .
 \end{align*}

Which takes care of the second requirement. Finally, we see that
 \begin{align*}
 (r_1 {\blacktriangleleft }_{tu} r_2 ){\blacktriangleleft }_{sv} r_3 & = (r_1 \triangleleft_{\operatorname{exp}(tu)} r_2) \triangleleft_{\operatorname{exp}(sv)} r_3\\
&= (r_1\triangleleft_{\operatorname{exp}(sv)} r_3) \triangleleft_{\operatorname{exp}(-sv)\operatorname{exp}(tu)\operatorname{exp}(sv)} (r_2\triangleleft_{\operatorname{exp}(sv)} r_3)
\end{align*}
 A~well-known result for Lie algebras, which is still true of Leibniz algebras, is that $ \operatorname{exp}(-sv)\allowbreak\times\operatorname{exp}(tu)\operatorname{exp}(sv) = \operatorname{exp}\bigl(t e^{s {\rm ad}_v}(u) \bigr) $. Consequently, we find that
 \begin{align*}
 (q_1 {\blacktriangleleft }_{tu} q_2 ){\blacktriangleleft }_{sv} q_3 &= (q_1\triangleleft_{\operatorname{exp}(sv)} q_3) \triangleleft_{\operatorname{exp}\bigl(t e^{s {\rm ad}_v}(u) \bigr) } (q_2\triangleleft_{\operatorname{exp}(sv)} q_3) \\
&= (q_1{\blacktriangleleft}_{sv} q_3) \blacktriangleleft_{t e^{s {\rm ad}_v}(u)} (q_2 {\blacktriangleleft }_{sv} q_3 ),
\end{align*}
and by setting $t=s=1$ we recover the third requirement for a~smooth Leibniz algebra family of smooth racks
\[ (q_1 {\blacktriangleleft }_{u} q_2 ){\blacktriangleleft }_{v} q_3 = (q_1{\blacktriangleleft}_{v} q_3) \blacktriangleleft_{ e^{ {\rm ad}_v}(u)} (q_2 {\blacktriangleleft }_{v} q_3 ).\]

Hence, for any smooth $G$-family of smooth racks $(R,\triangleleft)$, $F(Q,\triangleleft)$ is a~smooth $\mathfrak{g}$-family of smooth racks, and hence an object of $\textbf{LiFRk}$ since $\mathfrak{g}$ is a~Lie algebra. Now suppose $A\colon Q \to Q'$ together with $a\colon G\to G'$ is a~morphism from the smooth $G$-family of smooth racks $(R,\triangleleft)$ to the smooth $G'$-family of smooth racks $(R',\triangleleft')$, meaning that
$A(p\triangleleft_g q) = A(p) \triangleleft_{a(g)}' A(q)$.
 Define~${F(A,a) =(A,d_ea)}$, which we can see is a~morphism since for all $p,q \in R$ and $v \in \mathfrak{g}$, we have
 \begin{align*}
 A(p{\blacktriangleleft}_{v} q) &= A(p \triangleleft_{\operatorname{exp}_G(v)} q ) = A(p) \triangleleft_{a(\operatorname{exp}_G(v))} A(q) \\
 &= A(p) \triangleleft_{\operatorname{exp}_{G'}(d_ea(v))} A(q) = A(p) {\blacktriangleleft}_{d_ea(v)} A(q).
 \end{align*}

We note that $F(A,a)=(A,d_ea)$ leaves $A$ unchanged, and for the second component we have the chain rule and the fact that the differential of the identity is the identity, hence we can see that $F$'s morphism component preserves composition and identities. This allows us to conclude that $F$ is a~functor from the category of $\textbf{GpRk}$ to $\textbf{LiFRk}$.

We note that the morphism part of $F$ being defined by $F(A,a) =(A,d_ea)$ implies that $F$ is a~faithful functor when restricted to simply connected Lie group families of smooth racks, hence the category of smooth simply connected Lie group families of smooth racks can be considered a~category which is concrete over $\textbf{LFRk}$.

 Let $(R,\triangleleft)$ and $(R',\triangleleft')$ be smooth simply connected Lie group families of smooth racks, and let $\phi \colon R \to R'$ and $\varphi\colon \mathfrak{g}\to \mathfrak{g}'$ be a~morphism in $\textbf{LFRk}$ from $F(Q,\triangleleft) $ to $F(Q',\triangleleft')$. Hence there is a~Lie algebra homomorphism $\varphi\colon \mathfrak{g} \to \mathfrak{g}'$ such that
$
 \phi(q_1{\blacktriangleleft}_{v} q_2 ) = \phi(q_1){\blacktriangleleft'}_{\varphi(v)} \phi( q_2)$.
 To each Lie algebra, we can associate a~unique simply connected Lie group, we will call $\mathfrak{g}$'s Lie group $G$ (hence $(R,\triangleleft)$ is a~smooth $G$-family of smooth racks), and we will use $G'$ to refer to the Lie group associated to $\mathfrak{g}'$ (similarly, $(R',\triangleleft')$ is a~smooth $G'$-family of smooth racks). Moreover, to any Lie algebra homomorphism $\varphi\colon \mathfrak{g} \to \mathfrak{g}' $ there is a~unique Lie group homomorphism $\Phi\colon G \to G'$ such that $T_e \Phi = \varphi$. For any $q_1,q_2 \in Q$ and any $g \in G$, there is a~$v \in \mathfrak{g}$ such that $\operatorname{exp}_G(v) =g$, hence
 \begin{align*}
 \phi(q_1\triangleleft_g q_2) &= \phi(q_1{\blacktriangleleft}_v q_2) = \phi(q_1){\blacktriangleleft'}_{\varphi(v)} \phi(q_2) \\
 &
 = \phi(q_1)\triangleleft'_{\operatorname{exp}_{G'}(\varphi(v))}\phi(q_2) = \phi(q_1)\triangleleft'_{\Phi(g)}\phi(q_2).
 \end{align*}
 Therefore, $\phi$ is a~homomorphism from $(Q,\triangleleft)$ to $(Q',\triangleleft')$ in the category of smooth Lie group families of smooth racks. Hence we have that $F(\phi, \Phi ) = (\phi,\varphi)$ it follows that $F$ is a~full functor when restricted to the subcategory of all smooth simply connected Lie group families of smooth racks.

Consequently, if $(R,\triangleleft)$ and $(R',\triangleleft')$ are both smooth simply connected Lie group families of smooth racks such that $F(R,\triangleleft) = F(R',\triangleleft')$ and so that the identity maps 
 relating $F(R,\triangleleft) $ and~$ {F(R',\triangleleft')}$ both lie in the image of $F$'s morphism part when restricted to $\text{Mor}((R,\triangleleft),(R',\triangleleft'))$ and $\text{Mor}((R',\triangleleft'),(R,\triangleleft))$, then $R=R'$ and $\blacktriangleleft = {\blacktriangleleft }'$. As we have seen from the fullness argument, this implies $(R,\triangleleft)$ and $(R',\triangleleft')$'s operations are parameterized by the same simply connected Lie group $G$, and furthermore that the identity map $\operatorname{Id}\colon R \to R'$ determines a~morphism from $(R,\triangleleft)$ to $(R',\triangleleft')$. Hence $F$ is not only fully-faithful on the subcategory of smooth simply connected Lie group families of smooth racks, it is also strictly injective on objects, and so it is a~strong categorical embedding.
\end{proof}

Hence we can see explicitly that any smooth $G$-family of smooth racks (resp.\ smooth $G$-family of smooth quandles) can be thought of as a smooth Lie algebra family of smooth racks (resp.\ Lie quandles). The sufficient conditions for an object of $\textbf{LiFRk}$ to lie within the image of $F$ will be explored in a~sequel, where they will be used to classify a~particular family of smooth $G$-family of smooth racks. The strategy used will be distinct from, but heavily motivated by, the classification result found in the subsequent section.

\section[Classification of smooth R\^{}n-families of Alexander quandles on R\^{}m]{Classification of smooth $\boldsymbol{\mathbb{R}^n}$-families\\ of Alexander quandles on $\boldsymbol{\mathbb{R}^m}$} \label{Sec4}

 The goal of this section is a~classification of a~particular family of smooth Lie group families of smooth quandles, those being the smooth $\mathbb{R}^n$-families of Alexander quandles on $\mathbb{R}^m$. This collection of objects is precisely all of the pairs $\bigl(\mathbb{R}^m,\triangleleft\bigr)$ where the operation $\triangleleft$ is given by
\[ x \triangleleft_y x' = \alpha_y(x) + (\operatorname{Id} - \alpha_y)(x') \]
for a~homomorphism $y \mapsto \alpha_y \in {\rm GL}_m(\mathbb{R}) $ for each $y \in \mathbb{R}^n$. First, we'll need to establish a~concrete way to represent any $\mathbb{R}^n$-family of Alexander quandles on $\mathbb{R}^m$.

\begin{Lemma}
 Every $\mathbb{R}^n$-family of Alexander quandles on $\mathbb{R}^m$ $\bigl(\mathbb{R}^m,\triangleleft\bigr)$ has its operation given~by
 \[ x \triangleleft_y x' = e^{ \sum_{j=1}^n y_j u_j}(x) + \bigl(\operatorname{Id} - e^{ \sum_{j=1}^n y_j u_j}\bigr)(x') \]
 for some set $\{u_j\}_{j=1}^n $ of pairwise commuting $m$ by $m$ matrices.
\end{Lemma}

\begin{proof}
 Observing that for an $\mathbb{R}^n$-family of Alexander quandles on $\mathbb{R}^m$ $\bigl(\mathbb{R}^m,\triangleleft\bigr)$, the group law axiom for smooth $G$-families of smooth racks implies
 \begin{gather*}
\alpha_{y+y'}(x) + (\operatorname{Id} - \alpha_{y+y'})(x') \\
\qquad= x \triangleleft_{y+y'} x' = (x \triangleleft_yx' ) \triangleleft_{y'} x'
= \alpha_{y'}(x \triangleleft_{y} x' ) + (\operatorname{Id} - \alpha_{y'})(x') \\
\qquad=\alpha_{y'} \alpha_y(x) + (\alpha_{y'} - \alpha_{y'}\alpha_y)(x') + (\operatorname{Id} - \alpha_{y'})(x')
= \alpha_{y'} \alpha_y(x) + (\operatorname{Id} - \alpha_{y'}\alpha_y)(x').
\end{gather*}
We then obtain that $\alpha_{y'} \alpha_{y} = \alpha_{y+y'}$.
Hence \smash{$\alpha_y = e^{ \sum_{j=1}^n y_j u_j }$} for some collection $u=\{u_j\}_{j=1}^n \subset \mathfrak{gl}_m(\mathbb{R})$ %
of pairwise commuting $m$ by $m$ matrices, and we indicate this relationship between operators and sets of pairwise commuting matrices by using $\triangleleft^u$ to denote the operation.
\end{proof}

Armed with this lemma, we now have a~straightforward way to prove the following classification statement.

\begin{Theorem}
Two $\mathbb{R}^n$-families of Alexander quandles on $\mathbb{R}^m$, denoted $\bigl(\mathbb{R}^m,\triangleleft^u\bigr)$ and $\bigl(\mathbb{R}^m,\triangleleft^v\bigr)$, are isomorphic if and only if there are invertible linear maps $A\colon \mathbb{R}^m \to \mathbb{R}^m$ and $\alpha\colon \mathbb{R}^n \to \mathbb{R}^n$ such that
\[ \sum_{j=1}^n\alpha_{kj}u_j = Av_k A^{-1} \qquad\text{for}\quad k=1,2,\dots,n.
\]
\end{Theorem}

\begin{proof}
Now let us begin by supposing that there is a~$\textbf{GpRk}$ isomorphism $(F,f)\colon\bigl(\mathbb{R}^m,\triangleleft^v\bigr)\to \bigl(\mathbb{R}^m,\triangleleft^u\bigr)$. Then $F\colon \mathbb{R}^m \to \mathbb{R}^m$ is diffeomorphism, $f\colon \mathbb{R}^n \to \mathbb{R}^n$ is an invertible linear map, and we have that
\begin{gather*}
 e^{ \sum_{j=1}^nf_{j}(y) u_j }(F(x)) + \bigl(\operatorname{Id} - e^{\sum_{j=1}^nf_{j}(y) u_j }\bigr)(F(x')) \\
 \qquad= F(x) \triangleleft_{f(y)}^u F(x') = F\bigl(x \triangleleft_y^v x'\bigr)
 = F\bigl(e^{\sum_{j=1}^ny_j v_j }(x) + \bigl(\operatorname{Id} - e^{\sum_{j=1}^ny_j v_j }\bigr)(x') \bigr).
\end{gather*}
Differentiating the extremities of this equation with respect to each of the parameters $y_j$, we~find
\begin{gather*}
 d_0\bigl(e^{ \sum_{i,j=1}^nf_{ij}y_i u_j }(F(x)) + \bigl(\operatorname{Id} - e^{\sum_{i,j=1}^nf_{ij}y_i u_j }\bigr)(F(x'))\bigr)( \partial_{y_j} ) \\
 \qquad= \Bigg(\sum_{j=1}^n f_{kj} u_j \Bigg) (F(x) - F(x')) .
 \end{gather*}
Similarly,
\[ d_0\bigl(F\bigl(e^{\sum_{j=1}^ny_j v_j }(x) + \bigl(\operatorname{Id} - e^{\sum_{j=1}^ny_j v_j }\bigr)(x') \bigr) \bigr)(\partial_{y_k}) =(d_xF \circ v_k) (x - x'). \]
By re-equating these derivatives, we arrive at
\[ \Bigg(\sum_{j=1}^nf_{kj} u_j \Bigg) (F(x) - F(x')) = (d_xF\circ v_k)(x - x').
 \]
Now by differentiability of $F$, we must have that
\[
F(x')= F(x) +d_xF(x'-x) +o(x'-x).
 \]
Hence we arrive at
\[
 \Bigg(\sum_{j=1}^nf_{kj} u_j \Bigg) ( d_xF(y)) -d_xF\circ v_k(y) = \Bigg(\sum_{j=1}^nf_{kj} u_j \Bigg) ( o(y)).
\]
Where we have relabeled $y=x'-x$. Since the composition of any linear map with a~function which vanishes to at least first order at $0$ also vanishes to at least first order at $0$, the equation above implies
\[
\Bigg(\sum_{j=1}^nf_{kj} u_j\circ d_xF -d_xF\circ v_k \Bigg) (y) = o(y).
 \]
The left-hand side is a~linear function of $y$, but the right-hand side vanishes to at least first order at $0$. The zero map is the only linear function which vanishes to at least first order at $0$, and hence we must have that
\[
 \Bigg( \sum_{j=1}^nf_{kj} u_j\circ d_xF\Bigg)(y) = (d_xF\circ v_k)(y).
 \]
Because $x'$ (and consequently $y=x'-x$) is arbitrary, for every index $k$ we have that
\[
\sum_{j=1}^n f_{kj} u_j = d_xF \circ v_k\circ (d_xF)^{-1},
\]
which proves the first direction.

 On the other hand, let $A\colon \mathbb{R}^m \to \mathbb{R}^m$ and $\alpha\colon \mathbb{R}^n \to \mathbb{R}^n$ be invertible linear transformations such that $\sum_{j=1}^n \alpha_{kj} u_j = Av_kA^{-1}$ for pairwise commuting sets $\{v_j\}_{j=1}^n$, $\{u_j\}_{j=1}^n$ of $m$ by $m$ matrices. Then $A$ is a~diffeomorphism, $\alpha$ is an invertible linear map, and by the following calculation we see that they form the desired isomorphism
 \begin{align*}
A \bigl(x\triangleleft_y^vx' \bigr) & = A\circ e^{\sum_{j=1}^ny_j v_j}(x)+\bigl(A-A\circ e^{\sum_{j=1}^ny_j v_j}\bigr)(x') \\
&
= A\circ e^{\sum_{j=1}^ny_j v_j}\bigl(\bigl(A^{-1}\circ A \bigr)(x)\bigr) +\bigl(A -A \circ e^{\sum_{j=1}^ny_j v_j}\bigr)\bigl( \bigl(A^{-1} \circ A \bigr)(x')\bigr) \\
&
= A \circ e^{\sum_{j=1}^ny_j v_j}\circ A^{-1} (A(x)) +\bigl(\operatorname{Id}-A \circ e^{\sum_{j=1}^ny_j v_j}\circ A^{-1}\bigr)( A(x')) \\
&
= e^{ \sum_{j=1}^ny_j Av_j A^{-1}} (A(x))+\bigl(\operatorname{Id}- e^{ \sum_{j=1}^ny_j Av_j A^{-1}}\bigr)( A(x')) \\
&
= e^{\sum_{j=1}^n\alpha_j(y) u_j} (A(x))+\bigl(\operatorname{Id}- e^{ \sum_{j=1}^n\alpha_j(y) u_j}\bigr)( A(x')) \\
&
 = A(x) \triangleleft_{\alpha(y)}^u A(x').
\end{align*}
 The first equality above follows from the form of $\triangleleft^v_y$, the second follows from the definition of the inverse of a~function. The fourth equality follows from the conjugation identity for exponentials of matrices, the fifth from the assumed relationship between the $u_j$ and $v_j$ matrices, and the final equality follows from the form of $\triangleleft^u_y$. Hence we have that $\bigl(\mathbb{R}^m,\triangleleft^v\bigr)$ is isomorphic to $\bigl(\mathbb{R}^m,\triangleleft^u\bigr)$ as a~smooth $\mathbb{R}^n$-family of smooth quandles, as desired.
 \end{proof}

The machinery of this proof is heavily dependent on the fact that all differentials of smooth maps are linear, and hence Lie group homomorphisms between euclidean vector spaces. Consequently, extending this result to the generalized Alexander case on nonlinear groups is not straightforward. In a~sequel article, a~different strategy will be used to classify a~collection of smooth Lie group families of smooth racks which includes nonlinear examples.

\section{Noether's first theorem} \label{Sec5}

Fritz proposed a~generalization of Noether's first theorem~\cite{Fritz}, and suggested that connectedness may be a~sufficient hypothesis for this general version of the theorem in the category of Lie quandles. We will begin this section by demonstrating that connectedness is not a~necessary condition to imply that a~Lie quandle obey Noether's first theorem, though we say nothing about sufficiency here.

According to Fritz, Noether's first theorem
can be restated as a~certain condition satisfied by Lie quandles $(\mathfrak{g},\triangleleft)$ associated to Lie algebras, that being that for all $q,p \in Q$ we have that~${ q \triangleleft_t p = q}$ for all $t \in \mathbb{R}$ if and only if $ p \triangleleft_t q = p$.

This motivated Fritz to ask what hypotheses might imply that the condition holds for general Lie quandles. This theorem has an obvious generalization to the situation of a~smooth $G$-family of smooth racks $(R,\triangleleft)$, which is the statement that for all $p,q \in Q$ we have that $ p \triangleleft_g q = p $ for all $g \in G$ if and only if $q \triangleleft_g p =q $ for all $g \in G$.

It is clear already that we will need additional hypotheses on a~smooth $G$-family of smooth racks in order if we expect to prove Noether's first theorem in its context, since not even all Leibniz algebras possess this property. The necessary and sufficient conditions for a~Leibniz algebra to obey Noether's first theorem are given in the following proposition.

\begin{Proposition}
The smooth $\mathbb{R}$-family of smooth racks associated to a~$($finite-dimensional real$)$ Leibniz algebra $A$ obeys Noether's first theorem if and only if $[a,b]=0$ implies that $[b,a]=0$ for all $a,b \in A$.
\end{Proposition}

\begin{proof}
First, suppose that $[a,b]=0$ implies that $[b,a]=0$ for all $a,b \in A$, and that $e^{t \mathrm{ad}_b}(a) = a \triangleleft_t b = a $ for all $ t \in \mathbb{R} $. Then we must have that
\[
0 = \frac{\mathrm{d}}{\mathrm{d}t}\bigg|_{t=0}e^{t \mathrm{ad}_b}(a) = [a,b],
\]
hence $[b,a]=0$ by hypothesis. This then implies that
\[
b \triangleleft_t a = e^{t {\rm ad}_a}(b) = \sum_{j=0}^\infty \frac{t^j}{j!} {\rm ad}_a^j(b) = b.
\]
For the other direction, suppose $A$ obeys Noether's first theorem and $[a,b] =0$. Then it follows that $ a \triangleleft_t b = a $ for all $t \in \mathbb{R}$, and so $b\triangleleft_t a =b$ for all $t \in \mathbb{R}$. Thus we have that
\[
 0 = \frac{\mathrm{d}}{\mathrm{d}t}\bigg|_{t=0}e^{t \mathrm{ad}_a}(b) = [b,a]. \tag*{\qed}
 \]\renewcommand{\qed}{}
 \end{proof}

This result recovers~\cite[Example 3.5]{Fritz} when specialized to Lie algebras. The property above, that $[a,b]=0$ implies that $[b,a]=0$ for all $a,b \in A$, does not force the bracket of our Leibniz algebra to be either symmetric or anti-symmetric. This can be seen from the existence of nontrivial Lie brackets (since Lie brackets trivially obey this condition), and from the Leibniz algebra structure on $\mathbb{R}^2$ given by bilinearly extending $[e_1,e_1]=e_2$, $[e_1,e_2]=[e_2,e_1]=[e_2,e_2]=0$. Another way in which Fritz conjecture can be safely modified is by looking for a~hypothesis weaker than connectedness, as evidenced by the following proposition.

\begin{Theorem}
Noether's first theorem holds in all smooth $G$-families of smooth quandles of the form
$ f'H\triangleleft_{g} fH = f\rho_{g^{-1}}\bigl(f^{-1}f'\bigr)H$,
 where $F$ is a~Lie group, $H \subset F$ is a~closed Lie subgroup, and~${g \mapsto \rho_g \in \text{Aut}(F)}$ is a~$G$-action on~$F$ such that $\rho_g $ fixes $H$ pointwise for all $g \in G$, and so that for any $f\in F$ if $f^{-1}\rho_{g} (f) \in H$ for all $g \in G$ then $f \rho_{g}\bigl(f^{-1}\bigr) \in H$ for all $g \in G$ {\rm(}In particular, for every Lie group $F$ with $H=\{e\}$ and every Lie group action $ G \ni g \mapsto \rho_g\colon F \to F)$.
 \end{Theorem}

\begin{proof}
 Suppose that $f'H\triangleleft_{g} fH = f'H$ for all $g \in G$, meaning that
\[ f\rho_{g^{-1}}\bigl(f^{-1}f'\bigr) H =f'H\triangleleft_{g} fH= f'H.
\]
Left multiplying this equation by $(f')^{-1}$ we then find
\[ \bigl(f^{-1}f'\bigr)^{-1}\rho_{g^{-1}}\bigl(f^{-1}f'\bigr) H =(f')^{-1}f\rho_{g^{-1}}\bigl(f^{-1}f'\bigr) H = H,
\]
 for all $g \in G$. By hypothesis, we must have $ \bigl(f^{-1}f'\bigr)\rho_{g^{-1}}\bigl( (f')^{-1}f \bigr) \in H $
for all $g \in G$. Now by left multiplying the previous equation by $f$, we arrive at our desired conclusion
\[ fH\triangleleft_{g} f'H = f'\rho_{g^{-1}}\bigl((f')^{-1}f\bigr)H = fH.\tag*{\qed}
\]\renewcommand{\qed}{}
\end{proof}

Due to the existence of disconnected Lie groups such as the multiplicative group $\mathbb{R}\backslash\{0\}$, the theorem above shows that connectedness is at best a~sufficient condition to imply Noether's first theorem. We also would now like to demonstrate the sufficiency of a~different condition known in the quandle/rack literature as being ``faithful''.

\begin{Definition} A~rack $(R,\triangleleft)$ is said to be \textit{faithful} if the map $R \to \text{Aut}(R)$ defined by $ r \mapsto R_r $ (where $R_r(r') = r' \triangleleft r$) for all $r \in R$ is injective.
\end{Definition}

\begin{Proposition}
If $(R,\triangleleft)$ is a~smooth $G$-family of smooth racks having the property that $(R,\triangleleft_{g'}) $ is faithful for some $g' $ belonging to $G$'s center, then $(R,\triangleleft)$ has the first Noether property.
\end{Proposition}

\begin{proof}
 First, for a~bit of notation, we write $ R^g_y(x) = x \triangleleft_g y $. Generalized right distributivity and $g'$ in the center of $G$ 
 gives us that
\[ \bigl(R^{g'}_y \circ R^g_x\bigr)(a) =(a \triangleleft_g x)\triangleleft_{g'} y = (a \triangleleft_{g'} y)\triangleleft_{g} (x \triangleleft_{g'} y) =\bigl(R^{g}_{x \triangleleft_{g'} y} \circ R^{g'}_y\bigr)(a),\]
and hence we conclude that \smash{$R^{g'}_y \circ R^g_x=R^{g}_{x \triangleleft_{g'} y} \circ R^{g'}_y$}.
On the other hand, if $ x \triangleleft_{g}y =x$ for all $g \in G$, then we must also have that $ R_{x}^{h} = R_{x \triangleleft_{g}y}^{h} $ for all~${g,h \in G } $. In addition, we have that
\[\bigl(R_{x}^{g} \circ R^{g'}_y\bigr)(a) = (a \triangleleft_{g'} y ) \triangleleft_{g} x= (a \triangleleft_{ g } x) \triangleleft_{g'} (y \triangleleft_{g} x) = \bigl(R^{g'}_{y \triangleleft_g x} \circ R_{x}^g\bigr)(a),\]
from which we conclude that $R_{x}^{g} \circ R^{g'}_y =R^{g'}_{y \triangleleft_g x} \circ R_{x}^g $. Collecting all of our observations, we are led to the equation
\[ R_{x}^{g} \circ R^{g'}_y = R_{x\triangleleft_{g'} y}^{g} \circ R^{g'}_y =R^{g'}_y \circ R^g_x. \]
Upon right multiplying by $R^g_x$'s inverse, we then find that
\[ R^{g'}_{y \triangleleft_gx}= R_{x}^{g} \circ R^{g'}_y \circ \bigl(R^g_x\bigr)^{-1} =R^{g'}_y. \]
Hence by the faithfulness of $(R,\triangleleft_{g'})$, we must conclude that $y \triangleleft_g x = y$ for all $g \in G$, as desired.
\end{proof}

It is worth noting that faithfulness for a~central rack operation, like all other conditions heretofore discussed, is not a~necessary condition for the veracity of Noether's theorem. This can be seen from the fact that every trivial smooth $G$-family of smooth racks obeys Noether's first theorem, but any such family having more than one point (such as the trivial structure on~$\mathbb{R}$, for example) is not faithful for any group element.

\section{Discussion}

In the previous section, we demonstrated that none of the available hypotheses for Noether's first theorem are necessary conditions. Hence, more work is needed to identify sharp hypotheses. Both Fritz's and the present article leave the topic of Noether's second theorem, and its generalization to this context, completely unaddressed. Thorough investigation of Noether's second theorem in the context of Lie algebras and Lie quandles will be the subject of future work.

This article focuses mostly on the connection between smooth \(G\)-families of smooth racks/ quandles and Fritz's work.
Consequently, we have left further exploration of the extension of many knot invariants associated to $G$-families of quandles to the smooth setting for future work.

Finally, the small classification case dealt with here can be greatly strengthened. While the method used does not easily generalize to nonlinear examples, it turns out to be very suggestive of a~method that does. In a~sequel to this article, we will explore this alternative method and begin working towards a~more systematic classification of smooth $G$-families of smooth racks.

\subsection*{Acknowledgments}
The authors would like to thank Masahico Saito for his insight regarding $G$-families of quandles, as well as Tobias Fritz and Luc Ta for their helpful comments and questions. We would also like to thank the referees for their insight and helpful suggestions, which have greatly increased the readability and precision of this article.

\pdfbookmark[1]{References}{ref}
\LastPageEnding

\end{document}